\documentclass[11pt,letterpaper,reqno]{amsart}
\usepackage[includehead,includefoot,margin=30mm]{geometry}
\usepackage{epsfig}
\usepackage{amsfonts}
\usepackage{amsmath}
\usepackage{amssymb}
\usepackage{amsthm}
\usepackage{bbold}
\usepackage{float}
\usepackage{multicol}
\usepackage{times}
\usepackage{color}
\usepackage{enumerate}
\usepackage[T1]{fontenc}
\usepackage[utf8]{inputenc}
\usepackage{graphicx}
\usepackage{nicefrac}
\usepackage{xcolor}
\usepackage{textcomp}
\usepackage[all]{xy}
\usepackage{comment}
\usepackage{hyperref}
\hypersetup{
 colorlinks=true,
 linkcolor=blue,
 citecolor=blue,
 filecolor=blue,
 urlcolor=blue
}
\usepackage[normalem]{ulem}
\usepackage{pgfplots}
\pgfplotsset{compat=1.15}
\usepackage{mathrsfs}
\usetikzlibrary{arrows}

\usepackage{aliascnt}

\newtheorem{theorem}{Theorem}[section]
\newtheorem*{theorem*}{Theorem}

\newaliascnt{corollary}{theorem}
\newtheorem{corollary}[corollary]{Corollary}
\aliascntresetthe{corollary}

\newaliascnt{lemma}{theorem}
\newtheorem{lemma}[lemma]{Lemma}
\aliascntresetthe{lemma}

\newaliascnt{proposition}{theorem}
\newtheorem{proposition}[proposition]{Proposition}
\aliascntresetthe{proposition}

\newaliascnt{problem}{theorem}

\aliascntresetthe{problem}

\newaliascnt{conjecture}{theorem}

\aliascntresetthe{conjecture}

\theoremstyle{remark}
\newaliascnt{remark}{theorem}
\newtheorem{remark}[remark]{Remark}
\aliascntresetthe{remark}

\theoremstyle{definition}
\newaliascnt{definition}{theorem}

\aliascntresetthe{definition}

\newaliascnt{definitions}{theorem}

\aliascntresetthe{definitions}

\newaliascnt{notation}{theorem}

\aliascntresetthe{notation}

\newaliascnt{example}{theorem}
\newtheorem{example}[example]{Example}
\aliascntresetthe{example}

\newaliascnt{exercise}{theorem}

\aliascntresetthe{exercise}

\newaliascnt{sit}{theorem}

\aliascntresetthe{sit}

\usepackage[nameinlink,capitalize,noabbrev]{cleveref}

\crefname{theorem}{Theorem}{Theorems}
\crefname{corollary}{Corollary}{Corollaries}
\crefname{lemma}{Lemma}{Lemmas}
\crefname{proposition}{Proposition}{Propositions}
\crefname{problem}{Problem}{Problems}
\crefname{conjecture}{Conjecture}{Conjectures}
\crefname{claim}{Claim}{Claims}
\crefname{remark}{Remark}{Remarks}
\crefname{definition}{Definition}{Definitions}
\crefname{notation}{Notation}{Notations}
\crefname{example}{Example}{Examples}
\crefname{exercise}{Exercise}{Exercises}

\def\Z{\mathbb{Z}}

\def\M{\mathfrak{m}}

\def\XX{\mathbf{x}}

\newcommand{\KK}[0]{\ensuremath{\mathbf{k}}}

\newcommand{\ZZ}[0]{\ensuremath{\mathbb{Z}}}
\newcommand{\GA}[0]{\ensuremath{\mathbb{G}_{\mathrm{a}}}}
\newcommand{\GM}[0]{\ensuremath{\mathbb{G}_{\mathrm{m}}}}

\newcommand{\spec}[0]{\ensuremath{\operatorname{Spec}}}

\newcommand{\supp}[0]{\ensuremath{\operatorname{supp}}}

\newcommand{\Aut}[0]{\ensuremath{\operatorname{Aut}}}

\newcommand{\GL}[0]{\ensuremath{\operatorname{GL}}}

\newcommand{\Der}[0]{\ensuremath{\operatorname{Der}}}

\newcommand{\quot}[1]{{\overline{#1}}}

\makeatletter \newcommand*\bigcdot{\mathpalette\bigcdot@{.5}} \newcommand*\bigcdot@[2]{\mathbin{\vcenter{\hbox{\scalebox{#2}{$\m@th#1\bullet$}}}}} \makeatother

\usepackage{marvosym}

\begin{document}
\title[Finite-dimensional monomial algebras]{Finite-dimensional monomial algebras are determined by their automorphism group}

\author{Roberto D\'iaz}
\address{Departamento de Matem\'aticas, Facultad de Ciencias, Universidad de La Serena, Juan
Cisternas 1200, La Serena, Chile.}%
\email{roberto.diazv1@userena.cl}

\author{Giancarlo Lucchini Arteche}
\address{Departamento de Matem\'aticas, Facultad de Ciencias, Universidad de Chile, Las Palmeras 3425, \~{N}u\~{n}oa, Santiago, Chile.}%
\email{luco@uchile.cl}

%\date{\today}

\thanks{
 \Letter\, Giancarlo Lucchini Arteche \texttt{luco@uchile.cl}}

\begin{abstract}
A monomial algebra is the quotient of a polynomial algebra by an ideal generated by monomials. We prove that finite-dimensional monomial algebras are characterized by their automorphism group among finite-dimensional, local algebras with cotangent space of fixed dimension. In particular, we show how to recover a monomial ideal given the automorphism group of the corresponding monomial algebra.\\

\end{abstract}

\maketitle

%\tableofcontents

\section{Introduction}\label{Introduction}

A monomial algebra is the quotient of the polynomial algebra $\KK[\XX]=\KK[x_1,\ldots,x_n]$ by an ideal $I$ generated by monomials. These algebras have a natural action of the algebraic torus $T=\GM^n$ via $\mathbf{t}\cdot x_i=t_i x_i$, where $\mathbf t=(t_1,\ldots,t_n)$. This allows to study the algebra by means of combinatorial objects naturally associated to the action of a torus, in a very close analogy to the case of toric varieties. These combinatorial objects appear elsewhere in the literature creating links between several research areas. For instance, Stanley-Reisner rings are examples of monomial algebras that are related to abstract simplicial complexes \cite[\textsection 1.1]{MS05}, while discrete Hodge algebras form another family of monomial algebras related to Grassmannian varieties and their Schubert subvarieties \cite[\textsection\textsection 3--4]{DCEP82}.\\

The automorphism group of finite-dimensional monomial algebras was studied recently in \cite{DLMR24}, where it was proved that $\Aut_\KK(\KK[\XX]/I)$ is a linear algebraic group and that its neutral connected component $\Aut^0_\KK(\KK[\XX]/I)$ is generated by the torus $T$ and some generalized root groups associated to the adjoint action of $T$. Since $\KK[\XX]/I$ is finite-dimensional, the Lie algebra of $\Aut_\KK(\KK[\XX]/I)$ actually coincides with the algebra of derivations $\Der(\KK[\XX]/I)$ \cite[II, \textsection 4, 2.3]{DP70}, which is also studied in detail in \cite{DLMR24}. This Lie algebra is naturally related to the Lie algebras $\operatorname{Der}(\KK [\XX])$ of derivations of $\KK[\XX]$ and its subalgebra $\operatorname{Der}_I(\KK [\XX])$ of derivations that preserve the ideal $I$. These had been studied previously by several authors; see, for example, \cite{B95,T09,KLL15}.

In this paper, we use some of these results to prove that the automorphism group (and actually, its neutral connected component) of a finite-dimensional monomial algebra completely determines the algebra itself within a natural and relatively wide family of algebras. The main result is the following:

\begin{theorem}\label{thm main}
Let $\KK$ be a field of characteristic 0 and let $\KK[\XX]=\KK[x_1,\ldots,x_n]$. Let $\KK[\XX]/I$ be a finite-dimensional monomial algebra and let $(B,\M)$ be a finite-dimensional, local $\KK$-algebra with residue field $\KK$. Assume that $\dim_\KK(\M/\M^2)=n$ and that $\Aut^0_\KK(B)$ is isomorphic to $\Aut_\KK^0(\KK[\XX]/I)$ as algebraic $\KK$-groups. Then $B$ is isomorphic to $\KK[\XX]/I$.
\end{theorem}

Our main result follows the lines of several other results showing that certain algebraic objects are uniquely determined within a given category by their automorphism group. For instance, $\mathbb{P}^n$ is uniquely determined up to birational equivalence by its group of birational automorphisms in the category of irreducible varieties \cite{Can14,RUV24}. The affine space $\mathbb{A}^n$ is uniquely determined by its automorphism group, viewed as an ind-group, in the category of connected affine varieties \cite{Kra17}, by its automorphism group as an abstract group in the category of connected affine varieties \cite{cantat2019families}, and also by its automorphism group as an abstract group in the category of smooth irreducible quasi-projective varieties of dimension $n$ with nonvanishing Euler characteristic and finite Picard group \cite{kraft2021affine}. More generally, normal toric varieties non isomorphic to an algebraic torus are uniquely determined by their automorphism group in the category of irreducible normal affine varieties, and smooth affine spherical varieties non isomorphic to an algebraic torus are uniquely determined by their automorphism group in the category of irreducible smooth affine varieties \cite{LRU18,RV21b,RV24}. In the case of normal toric varieties, normality is essential, as shown in \cite{DLR23}, where the authors prove that if one drops the normality assumption, the variety is in general no longer determined by its automorphism group. The proof relies on a generalization, developed in \cite{DL24}, of the notion of Demazure roots introduced in \cite{Dem70} to affine toric varieties that are not necessarily normal.\\

Note that the assumptions in \cref{thm main} are indeed natural and to some extent necessary. Indeed, as we will see in \cref{example finiteness}, finite dimensionality is a necessary condition since $\KK[x]$, which is not finite-dimensional, has the same automorphism group as $\KK[x]/(x^3)$, which is a finite-dimensional monomial algebra. Now, monomial algebras are always local, while finite-dimensional algebras are always semi-local (i.e.~they have finitely many maximal ideals), so the reduction to the local case is not very restrictive. But this locality assumption is also necessary since, as we will show in \cref{example local}, $\KK[x]/(x^3)\oplus\KK[y]/(y^3)$ is a finite-dimensional (not local) algebra that has the same automorphism group as $\KK[x,y]/(x^2,y^2)$, which is finite-dimensional and monomial. The hypothesis on the residue field is also natural since the $\overline{\KK}$-algebra obtained after base change to an algebraic closure $\overline{\KK}$ of $\KK$ will be local if and only if its residue field was $\KK$ itself rather than a nontrivial finite extension of $\KK$. Moreover, the hypothesis is moot when $\KK$ is algebraically closed, which is most certainly the case of interest here.

The only hypothesis that seems really restrictive is the one on the dimension of $\M/\M^2$. However, if we look at these algebras geometrically, they represent zero-dimensional schemes (and more precisely, just a point) with a rich infinitesimal structure. In particular, the vector space $\M/\M^2$ corresponds to the cotangent space of the scheme, and hence gives an idea of the dimension of the infinitesimal geometric object being studied. In that sense, restricting oneself to objects ``of the same infinitesimal dimension'' makes sense. It would be interesting however to see if \cref{thm main} extends to the context of ``arbitrary infinitesimal dimension''. We have not been able to establish this.\\

In order to prove \cref{thm main}, we first prove that any algebra $B$ that satisfies the assumptions of the theorem must be monomial (\cref{Prop B is monomial}). Then, we need to prove that one can recover the monomial ideal $I$ from the automorphism group of the corresponding monomial algebra. This corresponds to the following statement, which does not use the hypothesis on the dimension of $\M/\M^2$.

\begin{theorem}\label{thm ideales y automorfismos}
Let $\KK$ be a field of characteristic 0, $A_1,A_2$ be finite-dimensional monomial $\KK$-algebras and assume that $\Aut^0_\KK(A_1)$ is isomorphic to $\Aut^0_\KK(A_2)$ as algebraic $\KK$-groups. Then $A_1$ is isomorphic to $A_2$. In particular, $A_1$ and $A_2$ can be presented as quotients of the same ring $\KK[\XX]$ by monomial ideals $I_1,I_2$ and, up to a change of variables sending $\XX^{m}$ to $\XX^{\varphi(m)}$ with $\varphi\in\GL_n(\Z)$, these two ideals coincide.
\end{theorem}

As a matter of fact, we will prove that if $I$ is a monomial ideal in $\KK[\XX]$, then we can recover $I$ from the Lie algebra of $\Aut_\KK(\KK[\XX]/I)$, i.e.~the algebra of derivations $\Der(\KK[\XX]/I)$, and its decomposition in weight spaces given by the adjoint action of $T$. Thus, \cref{thm ideales y automorfismos} can be seen as a complementary result to \cite{M95}, where it is proved that certain Stanley-Reisner algebras are determined by their algebra of derivations. Our main combinatorial result, \cref{lema combinatorio}, points towards a proof of this fact for an arbitrary monomial algebra, see \cref{remark: caso general}. In any case, the fact that the group of automorphisms can be replaced by its Lie algebra should not come as a surprise, given that these are linear groups by \cite[Proposition~4.1]{DLMR24} and every connected subgroup of $\GL_n(\KK)$ is uniquely determined by its Lie algebra when $\KK$ has characteristic 0, cf.~\cite[13.1]{Humphreys}.

\subsection*{Acknowledgements}
The authors would like to thank Andriy Regeta and Michel Brion for their comments on preliminary versions of this paper, the anonymous referees for their comments on the final version, and Gonzalo Manzano-Flores, Alvaro Liendo, Federico Castillo and Gary Mart\'inez-N\'u\~nez for interesting discussions.

\section{Preliminaries}\label{Section:preliminary}
Throughout the paper, $\KK$ will denote a field with characteristic zero and $n$ an element of $\Z_{>0}$. We fix a free $\ZZ$-module $M$ of rank $n$. We also define $N$ as the dual $\ZZ$-module $N=\operatorname{Hom}(M,\ZZ)$. There is a natural duality pairing
\[\langle\ ,\ \rangle \colon M\times N\to \ZZ,\quad  \mbox{defined by}\quad \langle m,p\rangle:=p(m)\,.\]
If we identify $M=\ZZ^n$ via the canonical basis $E=\{e_1,\ldots,e_n\}$ and $N=\ZZ^n$ via the dual basis $E^*=\{e^*_1\ldots,e^*_n\}$ of $E$, then the duality pairing simply becomes the standard scalar product. These two modules can and will be interpreted as the modules of characters and co-characters of the algebraic torus $T=\GM^n$, that is $M=\operatorname{Hom}(T,\GM)$ and $N=\operatorname{Hom}(\GM,T)$. The integer $p(m)$ for $m\in M$ and $p\in N$ can then be recovered as the integer such that $(m\circ p)(t)=t^{p(m)}$ for $t\in\GM$.

\subsection{Monomial ideals and monomial algebras}
In order to fix notations, we view the polynomial algebra $\KK[\XX]=\KK[x_1,\ldots,x_n]$ as a semigroup algebra $\KK[S]$ with respect to the monoid $S=\ZZ_{\geq 0}^n$. Recall that in all generality $\KK[S]$ is defined as
$$\KK[S]=\bigoplus_{m\in S}\KK\cdot\XX^m\quad\mbox{where}\quad \XX^m\cdot\XX^{m'}=\XX^{m+m'}\quad\mbox{and}\quad \XX^0=1\,.$$
By designating $\XX^{e_i}=x_i$, we obtain an isomorphism between $\KK[S]$ and $\KK[\XX]$, where $\XX^{m}$ represents the monomial $x_1^{m_1}\cdots x_n^{m_n}$. 

An ideal $I\subset \KK [\XX]$ is termed monomial if it has a generating set comprised of monomials, namely, $I=(\XX^{\mathbf{a}_1},\dots,\XX^{\mathbf{a}_l})$, with $\mathbf{a}_k\in \ZZ^n_{\geq 0}$. The support of a monomial ideal $I=(\XX^{\mathbf{a}_1},\dots,\XX^{\mathbf{a}_l})$ is the set
$$\supp(I)=\{m\in \ZZ^n_{\geq 0}\mid \mathbf{x}^m\in I\}=\bigcup_{k=1}^l (\mathbf{a}_k+\ZZ^n_{\geq 0}).$$
For practical reasons, we also define the co-support of $I$ as the complement of this set in $\Z_{\geq 0}^n$, that is
$$\supp^c(I)=\{m\in \ZZ^n_{\geq 0}\mid \mathbf{x}^m\not \in I\}.$$
A quotient $\KK[\XX]/I$ of a polynomial algebra by a monomial ideal is called a monomial algebra. The algebra $\KK[\XX]$ is naturally $\ZZ^n_{\geq 0}$-graded and under this grading every monomial ideal $I$ is a graded ideal. Therefore, the quotient $\KK[\XX]/I$ inherits a natural $\ZZ^n_{\geq 0}$-grading given by $\deg \quot{\XX}^m=m$ if $m\notin \supp(I),$ where $\quot{\XX}^m$ denotes the image of $\XX^m$ in $\KK[\XX]/I$. Recall that if $m\in \supp(I)$, then $\quot{\XX}^m=0$. In the sequel, we always regard $\KK[\XX]/I$ as a $\ZZ^n_{\geq 0}$-graded algebra. This grading naturally corresponds to an action of the torus $T=\GM^n$ given by $\mathbf{t}\cdot \XX^\alpha=\alpha(\mathbf{t})\XX^\alpha$.

Following \cite{DLMR24}, we say that a monomial ideal is full if $\XX^{e_i}\not\in I$ for every $1\leq i\leq n$. In this paper, \emph{we assume that all monomial ideals $I$ are full and that all monomial algebras are finite-dimensional as $\KK$-vector spaces}. In particular, $e_i\not\in\supp(I)$ for every $1\leq i\leq n$ and $\supp^c(I)$ is assumed to be finite.

\subsection{Derivations}\label{sec derivations}
A derivation $\partial$ on a $\KK $-algebra $B$ is a linear $\KK$-map satisfying the Leibniz rule, i.e., 
$$
\partial(fg)=\partial(f)g+f\partial(g), \quad \text{for all } f,g \in B.
$$
We denote by $\Der(B)$ the vector space of the derivations of $B$. Let $I\subset B$ be an ideal. We also denote by $\Der_I(B)\subset \Der(B)$ the subspace of derivations $\partial$ such that $\partial(I)\subset I$. Defining $[\partial,\partial']:=\partial\circ\partial'-\partial'\circ\partial$, both spaces are equipped with a Lie algebra structure.

Let $I$ be a monomial ideal on $\KK[\XX]$ and let $\partial\colon\KK[\XX]/I\to\KK[\XX]/I$ be a derivation.
Recall that $\KK[\XX]/I$ is $\ZZ^n_{\geq 0}$-graded. We say that $\partial$ is homogeneous if it sends homogeneous elements to homogeneous elements. By \cite[Lemma~1.1]{DL24} there exists a unique element $\alpha\in M=\ZZ^n$, called the degree of $\partial$, such that for every $\XX^m\notin\ker\partial$ we have $\partial(\XX^m)=\lambda\XX^{m+\alpha}$ for some $\lambda\in \KK^*$. We say that a homogeneous derivation $\partial$ is \emph{inner} if $\deg\partial\in \ZZ^n_{\geq 0}$ and \emph{outer} if $\deg\partial\in \ZZ^n\setminus \ZZ^n_{\geq 0}$. We apply the same terminology to the degrees themselves.

As we recalled in the Introduction, the Lie algebra $\Der(\KK[\XX]/I)$ is naturally related to the Lie algebra $\operatorname{Der}(\KK [\XX])$ and its subalgebra $\operatorname{Der}_I(\KK [\XX])$. These Lie algebras were studied for instance in \cite{B95,T09,KLL15,DLMR24}. We recall here some of these results.

First of all, by \cite[Theorem~2.2]{DLMR24}, we know that every derivation in $\Der(\KK[\XX]/I)$ comes from a derivation in $\Der_I(\KK[\XX])\subset\Der(\KK[\XX])$. In particular, homogeneous derivations come from homogeneous derivations, and these have a very particular shape. By \cite[Proposition~3.1]{KLL15}, every homogeneous derivation of degree $\alpha$ of $\KK[\XX]$ has the following form for a unique $p\in N_\KK$:
\begin{align}\label{equation:homogeneous}
\partial_{\alpha,p}\colon \KK[\XX]\to \KK[\XX]\quad\mbox{given by}\quad \XX^m\mapsto p(m)\XX^{m+\alpha}.
\end{align}
Moreover, if $\alpha$ is inner, then every derivation of $\KK[\XX]$ preserves any monomial ideal. In particular, the above formula defines a derivation of $\KK[\XX]/I$ for every $p\in N_\KK$.

On the other hand, by \cite[Theorem~2.7]{liendo2010affine}, if $\alpha$ is outer then the derivation of $\KK[\XX]$ inducing it must be of the form $\partial_{\alpha,\lambda e^*_k}$ with $1\leq k\leq n$, $\lambda\in \KK$ and $\alpha$ such that:
\begin{itemize}
    \item $e_k^*(\alpha)=-1$;
    \item $e_j^*(\alpha)\geq 0$ for every $j\neq k$;
    \item for every $\beta\in\supp(I)$, either $\alpha+\beta\in\supp(I)$ or $\alpha+\beta\not\in\Z_{\geq 0}^n$.
\end{itemize}
Note that the last condition, which ensures that the derivation preserves the ideal, can be checked by only looking at the generators of the ideal $I$.\\

In order to decide whether a derivation is nontrivial on the monomial algebra $\KK[\XX]/I$, we have the following result.

\begin{lemma}[\cite{DLMR24}, Lemma 3.9]\label{lema roberto}
The derivation $\overline\partial_{\alpha,p}$ of $\KK[\XX]/I$ induced by $\partial_{\alpha,p}\in\Der_I(\KK[\XX])$ is trivial if and only if
\[\alpha+(\ZZ_{\geq 0}^n\setminus p^\perp)\subset \supp(I).\]
\end{lemma}

\begin{example}\label{example}
In this example we show the $\alpha \in \ZZ^n$ such that $\partial_{\alpha,p}$ induces a nontrivial derivation for some $p\in N_\mathbf{k}$. In both \cref{fig:1-ex1,fig:2-ex1}, black bullets are the elements in $\supp(I_i)$, and green and red bullets are inner and outer degrees respectively for which there are nontrivial derivations.

\cref{fig:1-ex1} has the inner degrees $0$, $e_1$ and $e_2$ (in green) and the outer degrees $-e_1+2e_2$ and $2e_1-e_2$ (in red), while \cref{fig:2-ex1} only has inner degrees $0$, $e_1$ and $e_2$ (in green).

\begin{multicols}{2}

\centering

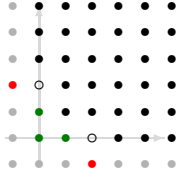
\begin{figure}[H]

\begin{picture}(100,80)
\definecolor{gray1}{gray}{0.7}
\definecolor{gray2}{gray}{0.85}
\definecolor{green}{RGB}{0,124,0}

\textcolor{gray2}{\put(0,20){\vector(1,0){60}}}

\textcolor{gray2}{\put(10,10){\vector(0,1){60}}}

\put(0,10){\textcolor{gray1}{\circle*{3}}}
\put(10,10){\textcolor{gray1}{\circle*{3}}}
\put(20,10){\textcolor{gray1}{\circle*{3}}} 
\put(30,10){\textcolor{red}{\circle*{3}}} 
\put(40,10){\textcolor{gray1}{\circle*{3}}} 
\put(50,10){\textcolor{gray1}{\circle*{3}}} 
\put(60,10){\textcolor{gray1}{\circle*{3}}}

\put(0,20){\textcolor{gray1}{\circle*{3}}}
\put(10,20){\textcolor{green}{\circle*{3}}} 
\put(20,20){\textcolor{green}{\circle*{3}}}
\put(30,20){\circle{3}} 
\put(40,20){\circle*{3}}
\put(50,20){\circle*{3}}
\put(60,20){\circle*{3}}

\put(0,30){\textcolor{gray1}{\circle*{3}}}
\put(10,30){\textcolor{green}{\circle*{3}}}
\put(20,30){\circle*{3}} 
\put(30,30){\circle*{3}} 
\put(40,30){\circle*{3}}
\put(50,30){\circle*{3}}
\put(60,30){{\circle*{3}}}

\put(0,40){\textcolor{red}{\circle*{3}}}
\put(10,40){{\circle{3}}}
\put(20,40){{\circle*{3}}}
\put(30,40){\circle*{3}} 
\put(40,40){{\circle*{3}}}
\put(50,40){{\circle*{3}}}
\put(60,40){\circle*{3}}

\put(0,50){\textcolor{gray1}{\circle*{3}}}
\put(10,50){{\circle*{3}}}
\put(20,50){{\circle*{3}}}
\put(30,50){{\circle*{3}}}
\put(40,50){\circle*{3}}
\put(50,50){\circle*{3}}
\put(60,50){\circle*{3}}

\put(0,60){\textcolor{gray1}{\circle*{3}}}
\put(10,60){{\circle*{3}}}
\put(20,60){{\circle*{3}}}
\put(30,60){{\circle*{3}}}
\put(40,60){\circle*{3}}
\put(50,60){\circle*{3}}
\put(60,60){\circle*{3}}

\put(0,70){\textcolor{gray1}{\circle*{3}}}
\put(10,70){{\circle*{3}}}
\put(20,70){{\circle*{3}}}
\put(30,70){{\circle*{3}}}
\put(40,70){{\circle*{3}}}
\put(50,70){\circle*{3}}
\put(60,70){\circle*{3}} 

\end{picture}
\caption{$I_1=(y^3,xy,x^3)$}\label{fig:1-ex1}
\end{figure}

\begin{figure}[H]
\begin{picture}(100,80)
\definecolor{gray1}{gray}{0.7}
\definecolor{gray2}{gray}{0.85}
\definecolor{green}{RGB}{0,124,0}

\textcolor{gray2}{\put(0,20){\vector(1,0){60}}}

\textcolor{gray2}{\put(10,10){\vector(0,1){60}}}

\put(0,10){\textcolor{gray1}{\circle*{3}}}
\put(10,10){\textcolor{gray1}{\circle*{3}}}
\put(20,10){\textcolor{gray1}{\circle*{3}}} 
\put(30,10){\textcolor{gray1}{\circle*{3}}}
\put(40,10){\textcolor{gray1}{\circle*{3}}} 
\put(50,10){\textcolor{gray1}{\circle*{3}}} 
\put(60,10){\textcolor{gray1}{\circle*{3}}}

\put(0,20){\textcolor{gray1}{\circle*{3}}}
\put(10,20){\textcolor{green}{\circle*{3}}} 
\put(20,20){\textcolor{green}{\circle*{3}}} 
\put(30,20){\circle*{3}}
\put(40,20){\circle*{3}}
\put(50,20){\circle*{3}}
\put(60,20){\circle*{3}}

\put(0,30){\textcolor{gray1}{\circle*{3}}}
\put(10,30){\textcolor{green}{\circle*{3}}} 
\put(20,30){\circle{3}}
\put(30,30){\circle*{3}}
\put(40,30){\circle*{3}}
\put(50,30){\circle*{3}}
\put(60,30){{\circle*{3}}}

\put(0,40){\textcolor{gray1}{\circle*{3}}}
\put(10,40){\circle*{3}}
\put(20,40){\circle*{3}}
\put(30,40){\circle*{3}}
\put(40,40){\circle*{3}}
\put(50,40){\circle*{3}}
\put(60,40){\circle*{3}}

\put(0,50){\textcolor{gray1}{\circle*{3}}} 
\put(10,50){\circle*{3}}
\put(20,50){\circle*{3}}
\put(30,50){\circle*{3}}
\put(40,50){\circle*{3}}
\put(50,50){\circle*{3}}
\put(60,50){\circle*{3}}

\put(0,60){\textcolor{gray1}{\circle*{3}}} 
\put(10,60){{\circle*{3}}} 
\put(20,60){{\circle*{3}}} 
\put(30,60){{\circle*{3}}} 
\put(40,60){\circle*{3}}
\put(50,60){\circle*{3}}
\put(60,60){\circle*{3}}

\put(0,70){\textcolor{gray1}{\circle*{3}}} 
\put(10,70){{\circle*{3}}}
\put(20,70){{\circle*{3}}}
\put(30,70){{\circle*{3}}}
\put(40,70){{\circle*{3}}}
\put(50,70){\circle*{3}}
\put(60,70){\circle*{3}} 

\end{picture}
\caption{$I_2=(y^2,x^2)$}\label{fig:2-ex1}
\end{figure}
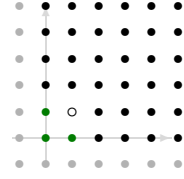

\end{multicols}

In both cases, the space of derivations of degree $0$ has dimension 2, and it corresponds to the Lie algebra of the torus $T$ acting on $\KK[\XX]/I_i$. Also in both cases, the spaces of derivations of degrees $e_1$ and $e_2$ each have dimension one. However, in $\KK[\XX]/I_1$ (\cref{fig:1-ex1}), the corresponding nontrivial derivations are induced by $\partial_{e_1,e_1^*}$ and $\partial_{e_2,e_2^*}$, while in $\KK[\XX]/I_2$ (\cref{fig:2-ex1}) the derivations are $\partial_{e_1,e_2^*}$ and $\partial_{e_2,e_1^*}$. One can check by hand that the corresponding Lie algebra of inner derivations are isomorphic as abstract Lie algebras.

\end{example}

\section{Recovering the monomial ideal $I$ from the automorphism group of $\KK[\XX]/I$.}\label{sec proof}
In this section we prove \cref{thm ideales y automorfismos}. Let $ I \subset \mathbb{Z}^n_{\geq 0} $ be a monomial ideal in $\KK[\XX]$ and let $G=\Aut_\KK(\KK[\XX]/I)$. We intend to recover the ideal $I$ assuming the group $G$ is given.

We start then by studying the dimension of the homogeneous components of the algebra of derivations. These correspond to the summands in the weight decomposition
\[\mathfrak g=\bigoplus_{\alpha\in\ZZ^n}\mathfrak g_\alpha,\]
given by the adjoint action of the torus $T\subset G$ on the Lie algebra $\mathfrak g$. For $\alpha\in\ZZ^n$, consider the following subset of the canonical basis $E\subset \Z^n$:
\begin{equation}\label{eqn def E_alpha}
E_\alpha:=\{e_i\mid \alpha+e_i\in\supp^c(I)\}.
\end{equation}
Note that this definition applies to outer $\alpha$'s as well. In that case, $E_\alpha$ can be nonempty only if there is a unique $1\leq k\leq n$ such that $\alpha_k=-1$ and $\alpha_j\geq 0$ for $j\neq k$.

\begin{lemma}\label{lema dimension g alpha}
Let $\alpha\in\Z^n$. Then the following hold:
\begin{itemize}
\item Assume that $\alpha$ is inner. Then $\dim\mathfrak{g}_\alpha=\sharp E_\alpha$.
\item Assume that $\alpha$ is outer, $\alpha+e_k$ is inner and $\dim\mathfrak{g}_{\alpha+e_k}=0$. Then $\dim\mathfrak{g}_{\alpha}=\sharp E_\alpha$, which is equal to $1$ if $\alpha+e_k\in\supp^c(I)$ and $0$ otherwise.
\end{itemize}
\end{lemma}

\begin{proof}
We prove the first assertion. Let $\alpha\in\Z_{\geq 0}^n$ and $p\in N_\KK$. Since $(\ZZ_{\geq 0}^n\setminus (e_i^*)^\perp)=e_i+\ZZ_{\geq 0}^n$, we see by \cref{lema roberto} that \begin{align*}
\overline\partial_{\alpha,e_i^*}=0 & \Leftrightarrow \alpha+(\ZZ_{\geq 0}^n\setminus (e_i^*)^\perp)\subset \supp(I) \\
& \Leftrightarrow \alpha+e_i+\ZZ_{\geq 0}^n\subset \supp(I)\\
&\Leftrightarrow \alpha+e_i\in\supp(I)\\
&\Leftrightarrow e_i\not\in E_\alpha.
\end{align*}
Since the $\overline\partial_{\alpha,e_i^*}$ generate $\mathfrak{g}_\alpha$ as a vector space, we only need to prove that the $\overline\partial_{\alpha,e_i^*}$ with $e_i\in E_\alpha$ are linearly independent in $\mathfrak{g}_\alpha$. Now, if $\sum_{e_i\in E_\alpha}\lambda_i \overline\partial_{\alpha,e_i^*}=0$, then $\overline\partial_{\alpha,p}=0$ with $p=\sum_{e_i\in E_\alpha}\lambda_i e_i^*$. In particular,
\[\overline\partial_{\alpha,p}(\overline{\XX}^{e_i})=p(e_i)\cdot\overline{\XX}^{\alpha+e_i}=\lambda_i\overline{\XX}^{\alpha+e_i}=0.\]
And since $\alpha+e_i\not\in\supp(I)$, we see that $\overline{\XX}^{\alpha+e_i}\neq 0$, so that $\lambda_i=0$ for every $i$. This proves that the $\overline\partial_{\alpha,e_i^*}$ are linearly independent.\\

We prove now the second assertion. Consider an outer $\alpha$ such that $\alpha+e_k$ is inner and $\dim\mathfrak{g}_{\alpha+e_k}=0$. Since $\alpha$ is outer, the only possible derivations with degree $\alpha$ are those induced by multiples of $\partial_{\alpha,e_k^*}\in\Der(\KK[\XX])$, as we recalled in \cref{sec derivations}. Thus, in order to check whether $\dim\mathfrak{g}_{\alpha}$ is $0$ or $1$, we need to check whether this derivation preserves $I$ and whether it is nontrivial on $\KK[\XX]/I$.

Assume first that $\alpha+e_k\in\supp(I)$. Then on one hand $\sharp E_{\alpha}=0$ since clearly $\alpha+e_j$ is outer for $j\neq k$ and hence not in $\supp^c(I)$. On the other hand, we have $\alpha+e_k+\beta\in\supp(I)$ for any $\beta\in\Z_{\geq 0}^n$, so that \cref{lema roberto} shows that $\mathfrak{g}_{\alpha}=0$ and the result holds in this case.

Assume from now on that $\alpha+e_k\in\supp^c(I)\subset\Z_{\geq 0}^n$. Then clearly that $E_\alpha=\{e_k\}$ and so $\sharp E_\alpha=1$. Moreover, since $\dim\mathfrak{g}_{\alpha+e_k}=0$ and $\alpha+e_k$ is inner, we know by the first part that $\sharp E_{\alpha+e_k}=0$, so that $\alpha+e_k+e_j\in\supp(I)$ for every $1\leq j\leq n$. Since $\partial_{\alpha,e_k^*}(\XX^{e_k})=\XX^{\alpha+e_k}$ and $\alpha+e_k\in\supp^c(I)$, we see that the derivation will be nontrivial on $\KK[\XX]/I$ as long as it is well defined, i.e.~as long as it preserves $I$.

We check this last assertion. Consider $\beta\in\supp(I)$. We have two different cases:
\begin{itemize}
\item If $\beta_k=0$, then
\[\partial_{\alpha,e_k^*}(\XX^{\beta})=e_k^*(\beta)\XX^{\alpha+\beta}=\beta_k\XX^{\alpha+\beta}=0\in I,\]
and we are done.
\item If $\beta_k>0$, since $\beta\neq e_k$, there exists $1\leq j\leq n$ such that $\beta-e_k-e_j\in\Z_{\geq 0}^n$ (note that we may have $j=k$). Then,
\[\partial_{\alpha,e_k^*}(\XX^{\beta})=\beta_k\XX^{\alpha+\beta}=\beta_k\XX^{(\alpha+e_k+e_j)+(\beta-e_k-e_j)}.\]
And since $(\alpha+e_k+e_j)\in\supp(I)$ and $(\beta-e_k-e_j)\in\Z_{\geq 0}^n$, we see that $\XX^{\alpha+\beta}\in I$, so the assertion holds in this case as well.
\end{itemize}
\end{proof}

With \cref{lema dimension g alpha} in mind, starting from the monomial ideal $I$, we associate an integer $m_I(\alpha)\in\{0,1,\ldots,n\}$ to every $\alpha\in\Z^n$ by putting $m_I(\alpha):=\sharp E_\alpha$. This is a purely combinatorial datum that allows us to recover the ideal, as the following result shows.\\

We say that a subset $C\subset\ZZ^n$ is bounded below if there exist $b_1,\ldots,b_n\in\Z$ such that, for every $\alpha\in C$, we have $\alpha_i\geq b_i$ for every $1\leq i\leq n$. Moreover, for any element $\beta\in\Z^n$, we will use the notation $C-\beta$ for the set $\{\alpha -\beta \in \mathbb{Z}^n \mid \alpha \in C\}$, that is, the set $C$ ``shifted by $-\beta$''.

\begin{proposition}\label{lema combinatorio}
Let $\mathcal{F}$ be the family of bounded below subsets $C\subset\ZZ^n$. Consider the map
$m:\mathcal F\to \{f:\Z^n\to\Z\}$
that sends $C\in \mathcal{F}$ to $m_C:=\sum_{i=1}^n \mathbb{1}_{C- e_i}$, that is, the function $m_C\colon \mathbb{Z}^n\to \Z$ defined by
\[\alpha\mapsto m_C(\alpha):=\sharp\{e_i\mid \alpha+e_i\in C\}.\]
Then the map $m$ is injective.
\end{proposition}

\begin{proof}
We will show how to recover the set $C$ from the function $m_C$. Since $C$ is bounded below, up to performing a translation of the function $m_C$, we may assume that $C\subset\Z_{\geq 0}^n$. Let $\alpha_1$ be the first coordinate of $\alpha$ and, for $a\in\Z_{\geq 0}$, denote by $C_a$ the subset $\{\alpha\in C\mid \alpha_1\leq a\}$. We will prove the following by induction on $a$:
\begin{itemize}
\item The function $\mathbb{1}_{C_a}$ is uniquely determined by $m_C$, and hence so is the function $m_{C_a}=\sum_{i=1}^n \mathbb{1}_{C_a-e_i}$. 
\item We have the formula
\begin{equation}\label{eqn formula recursiva m_C}
\mathbb{1}_{C_{a}}(\alpha)=\begin{cases}
m_C(\alpha-e_1)-m_{C_{\alpha_1-1}}(\alpha-e_1)& \text{if } 0\leq \alpha_1\leq a;\\
0 & \text{otherwise}.
\end{cases}
\end{equation}
\end{itemize}
Since $C\subset\Z^n_{\geq 0}$, for every $\alpha\in \Z^n$ we have $\mathbb{1}_C(\alpha)=\mathbb{1}_{C_a}(\alpha)$ for $a$ big enough (and depending on $\alpha$). Thus, the function $\mathbb{1}_{C}$ is uniquely determined by the functions $\mathbb{1}_{C_a}$ for $a\in\Z_{\geq 0}$, which are uniquely determined by $m_C$. This uniquely determines the set $C$, so that the map $m$ is indeed injective.\\

Let us start with $a=0$. If $\alpha_1=0$, then $\alpha-e_1+e_i\not\in\Z_{\geq 0}^n$ and hence $\alpha-e_1+e_i\not\in C$ for $i\geq 2$, which implies that $m_C(\alpha-e_1)\in\{0,1\}$. More precisely, for every such $\alpha$ we have
\[\alpha\in C\Leftrightarrow (\alpha-e_1)+e_1\in C\Leftrightarrow m_C(\alpha-e_1)=1.\]
In other words, we have
\[\mathbb{1}_{C_0}(\alpha)= \begin{cases}
   m_C(\alpha-e_1) & \text{if } \alpha_1=0;\\
   0 & \text{otherwise}.
\end{cases}\]
Recalling that $C_{-1}=\emptyset$ since $C\subset\Z_{\geq 0}^n$, this proves that the formula \eqref{eqn formula recursiva m_C} is valid in the case $a=0$.\\

Assume now that we have the induction hypothesis for a given $a$ and let us prove that
\begin{equation}\label{eqn formula recursiva m_C a+1}
\mathbb{1}_{C_{a+1}}(\alpha)=\begin{cases}
m_C(\alpha-e_1)-m_{C_{\alpha_1-1}}(\alpha-e_1)& \text{if } 0\leq \alpha_1\leq a+1;\\
0 & \text{otherwise}.\end{cases}
\end{equation}
First, note that for any $\alpha$ such that $\alpha_1\leq a$ or $\alpha_1\geq a+2$ we have $\mathbb{1}_{C_{a+1}}(\alpha)=\mathbb{1}_{C_{a}}(\alpha)$, and since $\mathbb{1}_{C_a}$ is determined by \eqref{eqn formula recursiva m_C}, we get the validity of \eqref{eqn formula recursiva m_C a+1} for such $\alpha$. Thus, we only need to prove the validity of \eqref{eqn formula recursiva m_C a+1} for $\alpha$ with $\alpha_1=a+1$.

Fix then $\alpha\in \Z^n$ such that $\alpha_1=a+1$. Then $\mathbb{1}_{C}(\alpha)=\mathbb{1}_{C_{a+1}}(\alpha)$. We also have
\[\alpha-e_1\in C \Leftrightarrow \alpha-e_1\in C_a,\]
and, more generally, for every $j\geq 2$,
\[\alpha-e_1+e_j\in C \Leftrightarrow \alpha-e_1+e_j\in C_a.\]
We have then two cases:
\begin{itemize}
    \item If $\alpha\not\in C$ (i.e. if $\mathbb{1}_{C_{a+1}}(\alpha)=0$), then $m_C(\alpha-e_1)=m_{C_a}(\alpha-e_1)$.
    \item If $\alpha\in C$ (i.e. if $\mathbb{1}_{C_{a+1}}(\alpha)=1$), then $m_C(\alpha-e_1)=m_{C_a}(\alpha-e_1)+1$.
\end{itemize}
Since here $a=\alpha_1-1$, this yields the formula for this particular $\alpha$ as well, concluding the proof.
\end{proof}

\begin{corollary}\label{corolario combinatorio}
Let $I$ be a monomial ideal and consider the function $m_I:\Z^n\to\Z$ defined by $m_\alpha:=\sharp E_\alpha$, where $E_\alpha$ is defined in \eqref{eqn def E_alpha}. Then $I$ is uniquely determined by the values of $m_I$ on every inner degree $\alpha$ and on those outer degrees $\alpha$ such that $\alpha+e_k$ is inner and $m_C(\alpha+e_k)=0$ for some $1\leq k\leq n$.
\end{corollary}

\begin{proof}
Using \cref{lema combinatorio}, it will suffice to prove that, when $C$ is the co-support of a monomial ideal $I$, one can obtain the remaining values of $m_C$ from those already given. Since these already given values include $m_C(\alpha)$ for every inner degree $\alpha$, we only need to consider an outer degree $\alpha$.

Now, if $\alpha+e_k$ is outer for $1\leq k\leq n$, then we know that $m_C(\alpha)=0$ for any co-support $C$, so we already know these values. If $\alpha+e_k$ is inner for some $1\leq k\leq n$, then such a $k$ is unique, so that $m_C(\alpha)\in\{0,1\}$ since $C\subset\Z_{\geq 0}^n$. Then we have two options:
\begin{itemize}
    \item either $m_C(\alpha+e_k)=0$ and the value $m_C(\alpha)$ is already given; or
    \item $m_C(\alpha+e_k)>0$, in which case we claim that $\alpha+e_k\in\supp^c(I)$ and hence $m_C(\alpha)=1$, which concludes the proof.
\end{itemize}
The last claim is easy to settle: if $\beta\in\supp(I)$, then $\beta+e_j\in\supp(I)$ for every $1\leq j\leq n$, so that $m_C(\beta)=0$. So $m_C(\beta)>0$ and $\beta$ inner implies $\beta\in\supp^c(I)$.
\end{proof}

Now we know how to associate a combinatorial datum, in the form of a function $m_I:\Z^n\to\Z$, to any monomial ideal $I$. On the other hand, starting from the linear algebraic group $G$ with maximal torus $T$ (recall from \cref{Section:preliminary} that the action of the torus $T=\GM^n$ comes naturally from the $\Z_{\geq 0}$-grading of the algebra), we may associate to every $\alpha\in\Z^n$ the number $\dim(\mathfrak{g}_\alpha)$, which is related to $m_I$ thanks to \cref{lema dimension g alpha}. This is all we need in order to prove \cref{thm ideales y automorfismos}.

\begin{proof}[Proof of \cref{thm ideales y automorfismos}]
Let $G_i=\Aut_\KK(A_i)$ for $i=1,2$ and let $\varphi:G_1\to G_2$ denote an isomorphism between $G_1$ and $G_2$. Since this is an isomorphism between linear algebraic groups, it sends maximal tori to maximal tori. In particular, the dimension of these tori coincide. Now, as it is noted in \cite[Lemma 4.5]{DLMR24}, the dimension of a maximal torus of $G_i$ coincides with the number of variables in the polynomial ring from which we obtain $A_i$. Thus, we may assume that both $A_1$ and $A_2$ are quotients of the polynomial ring $\KK[\XX]$ in $n$ variables by respective monomial ideals $I_1$ and $I_2$.

Denote by $T_i$ the maximal torus of $G_i$ acting on $\KK[\XX]/I_i$ in the standard way and denote by $M_i$ the corresponding group of characters. The torus $\varphi(T_1)\subset G_2$ is maximal, hence conjugate to $T_2$ by an element $g\in G_2$. Up to replacing $\varphi$ by its composition with the conjugation by $g$, we may assume that $\varphi$ sends $T_1$ to $T_2$. In particular, we have an induced isomorphism $\varphi^*:M_2\to M_1:\beta\mapsto \beta\circ\varphi$. This shows that up to a change of variables on one algebra given by $\varphi^*$, we may identify $T_1$ and $T_2$ with a single torus $T$ acting on $\KK[\XX]$ in the usual way, and hence we may identify both $M_1$ and $M_2$ with $\Z^n$ in the usual way.

Consider now the isomorphism of Lie algebras $\phi:\mathfrak{g}_1\to\mathfrak{g}_2$ induced by $\varphi$. Since $\varphi$ is an isomorphism of algebraic groups that induces the identity on $T$, we see that $\phi$ is a $T$-equivariant linear isomorphism with respect to the adjoint actions. In particular, $\phi$ induces isomorphisms $\phi_\alpha:\mathfrak{g}_{1,\alpha}\to\mathfrak{g}_{2,\alpha}$ for every $\alpha\in \Z^n$. This obviously shows that $\dim(\mathfrak{g}_{1,\alpha})=\dim(\mathfrak{g}_{2,\alpha})$ for every $\alpha\in\Z^n$, and hence, by \cref{lema dimension g alpha} and \cref{corolario combinatorio}, we know that both data come from the same monomial ideal $I\subset\KK[\XX]$. In other words, up to the change of variables we did previously, $I_1$ and $I_2$ coincide.
\end{proof}

\begin{remark}
The information supplied in \cref{corolario combinatorio} by the outer degrees is strictly necessary in general. Indeed, one can immediately check that the monomial ideals $I_1$ and $I_2$ given in \cref{example} give away the same function $m_I$ when restricted to inner degrees: we have $m_{I}(0)=2$, $m_{I}(e_j)=1$ for $j=1,2$ and $m_{I}(\alpha)=0$ otherwise if $\alpha\in\Z_{\geq 0}^n$. In particular, as it was stated therein, the corresponding Lie algebras of inner derivations are isomorphic.
\end{remark}

\begin{remark}\label{remark: caso general}
If we consider a general monomial ideal $I\subset \KK[\XX]$ (i.e.~with $\KK[\XX]/I$ not necessarily finite-dimensional) and replace $\mathfrak{g}$ with the graded Lie algebra of derivations $\Der(\KK[\XX]/I)$, then \cref{lema dimension g alpha}, and \cref{corolario combinatorio} hold with the exact same proofs. This shows that the graded algebra of derivations determines the monomial algebra in all generality.
\end{remark}

\section{Counterexamples and proof of the \cref{thm main}}
In this last section we focus on \cref{thm main}. As we claimed in the Introduction, the hypothesis of $B$ being a finite-dimensional $\KK$-algebra is necessary in order to obtain the result, as the following example shows.

\begin{example}\label{example finiteness}
Consider the infinite-dimensional $\KK$-algebra $\KK[x]$. An automorphism $\varphi$ of $\KK[x]$ is uniquely determined by the image of $x$. Since $\varphi$ is bijective, there exists $P\in\KK[x]$ such that $\varphi(P)=P(\varphi(x))=x$. And since $\deg(P(\varphi(x)))=\deg(P)\deg(\varphi(x))$, we see that $\deg(\varphi(x))$ must be $1$. In other words, $\varphi(x)$ must be a linear polynomial $ax+b$ with $a\in\KK^*$ and $b\in\KK$. Direct computations show that automorphisms of the form $x\mapsto ax$ define a subgroup isomorphic to $\GM$, while those of the form $x\mapsto x+b$ define a subgroup isomorphic to $\GA$. One last direct computation shows then that the automorphism group of $\KK[x]$ is isomorphic to $\GA\rtimes\GM$, where $\mathbb{G}_m$ acts on $\mathbb{G}_a$ with weight $\alpha=1$.

On the other hand, the automorphism group of the monomial algebra $\KK[x]/(x^3)$ can be described either in a similar explicit fashion, where one discovers that an automorphism sends $\bar x$ to $a\bar x+b\bar x^2$ with $a\in\KK^*$ and $b\in\KK$, or by looking at \cref{fig:ex-2} and following the theory developed in \cite[Section 4]{DLMR24}. If $I=(x^3)$, then $\supp(I)=\{3,4,\ldots\}$ (the black bullets in \cref{fig:ex-2}) and $\supp^c(I)=\{0,1,2\}$, and hence the function $m_I$ gives $1$ if $\alpha=0,1$ (the green bullets in \cref{fig:ex-2}) and $0$ otherwise. Following \cite[Example 4.24]{DLMR24}, the automorphism group is connected and generated by the torus $T = \GM$, whose Lie algebra corresponds to the weight $0$ subspace in the algebra of derivations, and the root group $U_\alpha = \GA$, whose Lie algebra corresponds to derivations with weight $\alpha=1$. Explicitly, $T$ acts by $\bar x \mapsto a \bar x$ and $U_\alpha$ by $\bar x \mapsto \bar x + b \bar x^2$, for $a\in\KK^*$ and $b\in\KK$. A direct computation confirms then that the action of $\GM$ on $\GA$ by conjugation has weight 1.

\begin{figure}[H]
\begin{picture}(100,20)
\definecolor{gray1}{gray}{0.7}
\definecolor{gray2}{gray}{0.85}
\definecolor{green}{RGB}{0,124,0}

\textcolor{gray2}{\put(0,20){\vector(1,0){80}}}

\put(0,20){\textcolor{gray1}{\circle*{3}}}
\put(10,20){\textcolor{gray1}{\circle*{3}}}
\put(20,20){\textcolor{green}{\circle*{3}}} 
\put(30,20){\textcolor{green}{\circle*{3}}} 
\put(40,20){\circle{3}} 
\put(50,20){\circle*{3}} 
\put(60,20){\circle*{3}}
\put(70,20){\circle*{3}}

\put(0,0){$-1$}
\put(18,0){$0$}
\put(28,0){$1$}
\put(38,0){$2$}
\put(48,0){$3$}
\put(58,0){$4$}
\put(68,0){$5$}

\end{picture}
\caption{$I=(x^3)$}\label{fig:ex-2}
\end{figure}

All in all, we have an isomorphism $ \Aut_\KK(\KK[x]) \simeq \Aut_\KK(\KK[x]/(x^3))$, yet $\KK[x]$ is clearly not isomorphic to $ \KK[x]/(x^3)$.
\end{example}

In the Introduction we also claimed that the algebra $B$ in \cref{thm main} needs to be local. The following example shows why.

\begin{example}\label{example local}
Consider the finite-dimensional $\KK$-algebra $B_1 = \KK[x]/(x^3) \oplus \KK[y]/(y^3) $ and the finite-dimensional monomial $\KK$-algebra $ B_2=\KK[x, y]/(x^2, y^2) $ (see \cref{fig:2-ex1}). The automorphism group of the latter is computed in \cite[Example 4.25]{DLMR24}, but we give some details here for the comfort of the reader. As we remarked in \cref{example}, and according to \cite[Theorem 4.22]{DLMR24}, this group has two root groups corresponding to the roots $e_1$ and $e_2$, which commute since the corresponding derivations commute. Thus, they generate a subgroup isomorphic to $\GA^2$ on which $T=\GM^2$ acts coordinatewise, with weight one on each coordinate, so that
\[\Aut_\KK^0(B_2)\simeq \GA^2\rtimes\GM^2=(\GA\rtimes\GM)^2.\]
Moreover, \cite[Theorem 4.22]{DLMR24} shows that the finite quotient $\Aut_\KK(B_2)/\Aut_\KK^0(B_2)$ is isomorphic to $\Z/2\Z$ and it is given by the permutation of the variables, which corresponds to the permutation of the two copies of $(\GA\rtimes\GM)$. Hence,
\[\Aut_\KK(B_2)\simeq (\GA\rtimes\GM)^2\rtimes\Z/2\Z.\]

On the other hand, since the two components of $B_1$ are isomorphic, it is easy to see that
\[\Aut_\KK(B_1)=(\Aut_\KK(\KK[x]/(x^3)))^2\rtimes\Z/2\Z,\]
where $\Z/2\Z$ permutes the two components of the direct sum. Now, we already computed $\Aut_\KK(\KK[x]/(x^3))$ in \cref{example finiteness}, so we finally get that
$\Aut_\KK(B_1) \simeq (\GA\rtimes\GM)^2\rtimes\Z/2\Z\simeq \Aut_\KK(B_2)$. But again, they are not isomorphic as algebras, since one of them is local and the other is not.
\end{example}

These examples show that we must focus on finite-dimensional, local $\KK$-algebras. The last hypothesis in \cref{thm main}, that is, the dimension of the cotangent space $\M/\M^2$ being $n$, is crucial in order to compare such an algebra with a monomial algebra. This is the subject of our last result.

\begin{proposition}\label{Prop B is monomial}
Let $\KK$ be a field of characteristic $0$, \( (B, \mathfrak{m}) \) a local $\KK$-algebra with residue field $\KK$ and \( \dim(\mathfrak{m}/\mathfrak{m}^2) = n \), and let $\KK[\XX]=\KK[x_1,\ldots,x_n]$. Assume that \( \Aut_\KK(B) \simeq \Aut_\KK(\KK[\XX]/I) \) as algebraic $\KK$-groups, where $\KK[\XX]/I$ is a finite-dimensional monomial algebra. Then \( B \) is a monomial algebra.
\end{proposition}

Note that, together with \cref{thm ideales y automorfismos}, this result immediately yields \cref{thm main}, so we are only left with one last proof.

\begin{proof}[Proof of \cref{Prop B is monomial}]
Since \( \dim(\mathfrak{m}/\mathfrak{m}^2) = n \), we know by Nakayama's Lemma \cite[Corollary~4.8]{E95} that any minimal generating set of $B$ as a $\KK$-algebra has exactly $n$ elements. We may write then $B=\KK[\XX]/J$ for a certain ideal $J$. We claim that the automorphism group of such an algebra is a linear algebraic group. Indeed, it acts faithfully and linearly on the finite-dimensional space $B$, so that it is isomorphic to a subgroup $G$ of $\GL(B)$, which is isomorphic to $\GL_m(\KK)$ for $m=\dim(B)$. It suffices then to prove that $G$ is Zariski closed in $\GL(B)$. Now, every automorphism of $B$ corresponds to an endomorphism of $\KK[\XX]$ sending $J$ to itself. This implies that an element $\varphi\in \GL(B)$ belongs to $G$ if and only if, for every polynomial $Q\in J$, we have $Q(\varphi(\bar x_1),\ldots,\varphi(\bar x_n))=0\in B$. Each polynomial $Q\in J$ defines thus a Zariski-closed condition on an element of $\GL(B)$ to belong to $G$. Since we are taking the intersection of the subgroups of $\GL_m(\KK)$ satisfying all of these conditions, we see that this is indeed a Zariski-closed subgroup of $\GL(B)$.

Now, since \( \Aut_\KK(B) \simeq \Aut_\KK(\KK[\XX]/I) \), the dimension of their maximal tori coincide. In particular, we know that the torus $T=\GM^n$ acts faithfully on $B$, inducing a $\Z^n$-grading on $B$. In other words, $B=\bigoplus_{\alpha\in\Z^n} B_\alpha$. Up to decomposing the generators $\bar x_i\in B$ into homogeneous pieces with respect to this grading and recalling that any minimal generating set of $B$ as a $\KK$-algebra has exactly $n$ elements, we may assume that $\bar x_i\in B_{\alpha_i}$ for some $\alpha_i\in\Z^n$.

In order to conclude, by \cref{lema accion fiel} below, we know that the set of those $\alpha\in\Z^n$ with $B_\alpha\neq 0$ must generate the whole group $\Z^n$. Now, since the $\bar x_i$ generate $B$, we see that all such $\alpha$ must be generated by the $\alpha_i$. In other words, $\{\alpha_1,\ldots,\alpha_n\}$ is a $\Z$-basis of $\Z^n$. Then, as in \cref{thm ideales y automorfismos}, up to a change of variables sending $\XX^{m}$ to $\XX^{\varphi(m)}$ with $\varphi\in\GL_n(\Z)$, we may assume that $\alpha_i=e_i$, in which case we see that $B=\KK[\XX]/J$ with the usual action of $T$, i.e.~$B$ is monomial \cite[Proposition~2.1]{MS05}, as wished.
\end{proof}

\begin{lemma}\label{lema accion fiel}
Let $\KK$ be a field of characteristic $0$ and let $B$ be a $\KK$-algebra. Let $T$ be a split $\KK$-torus acting on $\spec(B)$ and let $B=\bigoplus_{\alpha\in\Z^n}B_\alpha$ be the corresponding $\Z^n$-grading. Then the action of $T$ is faithful if and only if the set $\{\alpha\in\Z^n\mid B_\alpha\neq 0\}$ generates $\Z^n$.
\end{lemma}

\begin{proof}
An element $\mathbf{t}\in T$ acts trivially on $B$ if and only if $\mathbf{t}\in\ker(\alpha)$ for every $\alpha\in\Z^n$ such that $B_\alpha\neq 0$. Thus, the action is faithful if and only if $\bigcap_{B_\alpha\neq 0}\ker(\alpha)=0$. And this will occur if and only if these characters generate the whole group of characters of $T$.
\end{proof}

\section*{Declarations}

\textbf{Funding}: The first author was supported by ANID via Proyecto Fondecyt Postdoctorado N\textsuperscript{o}3230406. The second author was partially supported by ANID via Proyecto Fondecyt Regular N\textsuperscript{o}1240001. Both authors were partially supported by ECOS-ANID (Grant N\textsuperscript{o}ECOS230044).\\

{\bf Conflicts of Interest:} The authors declare that they have not conflict of interest.\\

{\bf Data Availability Statement:} Not applicable.

\bibliographystyle{alpha}
\bibliography{ref}
\end{document}